\documentclass[12pt]{amsart}
\usepackage{amssymb}
\usepackage[utf8]{inputenc}
\usepackage{enumerate}

\def\eps{\varepsilon }

\def\RR{\mathbb R}

\def\bb{\mathcal B}

\def\uu{\mathcal U}

\def\ww{\mathcal W}

\def\uuu{\overline{\mathcal U}}
\def\uuuq{\overline{\mathcal U_q}}
\def\vv{\mathcal V}

\newcommand{\set}[1]{\left\lbrace #1\right\rbrace}
\providecommand{\abs}[1]{\left\lvert#1\right\rvert}

\newcommand{\remove}[1]{ }
\DeclareMathOperator{\dime}{dim}

\newtheorem{theorem}{Theorem}[section]
\newtheorem{proposition}[theorem]{Proposition}
\newtheorem{lemma}[theorem]{Lemma}

\theoremstyle{definition}

\theoremstyle{remark}
\newtheorem*{remark}{Remark}

\numberwithin{equation}{section}

\usepackage[english]{babel}

\begin{document}
\title[Hausdorff dimension of univoque sets]{Hausdorff dimension of univoque sets and Devil's staircase}
\author[V. Komornik]{Vilmos Komornik}
\address{Département de mathématique\\
         Université de Strasbourg\\
         7 rue René Descartes\\
         67084 Strasbourg Cedex, France}
\email{komornik@math.unistra.fr}
\author[D. Kong]{Derong Kong}
\address{School of Mathematical Science, Yangzhou University,
    Yangzhou, JiangSu 225002, People's Republic of China}
\email{derongkong@126.com}
\author[W. Li]{Wenxia Li}
\address{Department of Mathematics, Shanghai Key Laboratory of PMMP, East China Normal University, Shanghai 200062,
People's Republic of China}
\email{wxli@math.ecnu.edu.cn}
\thanks{}
\date{Version of 2015-03-02-a}

\begin{abstract}
We fix a positive integer $M$, and we consider expansions in arbitrary real bases $q>1$ over the alphabet $\set{0,1,\ldots,M}$.
We denote by $\uu_q$ the set of real numbers having a unique expansion.
Completing many former investigations, we give a formula for the Hausdorff dimension $D(q)$ of $\uu_q$ for each $q\in (1,\infty)$.
Furthermore, we prove that the dimension function $D:(1,\infty)\to[0,1]$ is continuous, and has a bounded variation.
Moreover, it has a Devil's staircase behavior in $(q',\infty)$, where $q'$ denotes the Komornik--Loreti constant: although $D(q)>D(q')$ for all $q>q'$, we have
$D'<0$ a.e. in $(q',\infty)$.
During the proofs we improve and generalize a theorem of Erd\H os et al. on the existence of large blocks of zeros in  $\beta$-expansions, and
we determine  for all $M$ the Lebesgue measure and the Hausdorff dimension of the set $\uu$ of bases in which $x=1$ has a unique expansion.
\end{abstract}
\maketitle

\section{Introduction}\label{s1}

Fix a positive integer $M$ and an \emph{alphabet} $\set{0,1,\ldots,M}$.
By a \emph{sequence}  we mean an element $c=(c_i)$ of $\set{0,1,\ldots,M}^{\infty}$.

Given a real \emph{base} $q>1$, by an \emph{expansion} of a real number $x$ we mean a sequence $c=(c_i)$ satisfying the equality
\begin{equation*}
\pi_q(c):=\sum_{i=1}^{\infty}\frac{c_i}{q^i}=x.
\end{equation*}

Expansions of this type in \emph{non-integer} bases have been extensively investigated since a pioneering paper of Rényi \cite{Renyi1957}.
One of the striking features of such bases is that generically a number has a continuum of different expansions, a situation quite opposite to that of integer bases; see, e.g., \cite{ErdosJooKomornik1990} and Sidorov \cite{Sid2003a}.
However, surprising unique expansions have also been discovered by Erdős et al. \cite{ErdosHorvathJoo1991}, and they have stimulated many works during the last 25 years.

We refer to the papers \cite{KomornikLoreti2007}, \cite{DeVries2008}, 
\cite{DeVries2009}, \cite{DeVriesKomornik2009}, \cite{DeVriesKomornik2010}, \cite{Baker2012} and surveys \cite{Sidorov2003}, \cite{Komornik2011} and \cite{DeVriesKomornik2014} for more information.

Let us denote by $\uu_q$ the set of numbers $x$ having a unique expansion and by $\uu'_q$ the set of the corresponding expansions.
The topological and combinatorial structure of these sets have been described in \cite{DeVriesKomornik2009}.
The present paper is a natural continuation of this work, concerning the measure-theoretical aspects.

Daróczy and Kátai \cite{DaroczyKatai1995} have determined the Hausdorff dimension of $\uu_q$ when $M=1$ and $q$ is a Parry number.
Their results were extended by Kallós and Kátai \cite{Kallos1999}, \cite{Kallos2001}, \cite{KataiKallos2001}, Glendinning and Sidorov \cite{GlendinningSidorov2001}, Kong et al. \cite{KongLiDekking2010},  \cite{KongLi2014}, and in \cite{DeVriesKomornik2010}, \cite{BaatzKomornik2011}.

We recall from \cite{KomornikLoreti1998} and \cite{KomornikLoreti2002} that there exists a smallest base $1<q'<M+1$ (depending on $M$) in which $x=1$ has a unique expansion: the so-called \emph{Komornik--Loreti constant}.

We also recall two theorems on the \emph{dimension function}
\begin{equation*}
D(q):=\dim_H\uu_q,\quad 1<q<\infty,
\end{equation*}
obtained respectively in  \cite{GlendinningSidorov2001}, \cite{KongLiDekking2010} and  in \cite{KongLi2014}:

\begin{theorem}\label{t11}
The  function $D$ vanishes in $(1,q']$, and $D>0$ in $(q',\infty)$.
Its maximum $D(q)=1$ is attained only in $q=M+1$.
\end{theorem}

It follows from this theorem that $\uu_q$ is a (Lebesgue) null set for all $q\ne M+1$, while $\uu_{M+1}\subseteq [0,1]$ has  measure one because its complementer set is countable in $[0,1]$.
Since $\uuuq\setminus\uu_q$ is countable for each $q$ (see \cite{DeVriesKomornik2009}), the same properties hold for $\uuuq$ as well.

\begin{theorem}\label{t12}
For almost all $q>1$, $\uu_q'$ is a subshift, and
\begin{equation}\label{11}
D(q)=\frac{h(\uu_q')}{\log q},
\end{equation}
where $h(\uu_q')$ denotes the topological entropy of $\uu_q'$.

Furthermore, the function $D$ is differentiable almost everywhere.
\end{theorem}

We recall from Lind and Marcus \cite{LindMarcus1995} that
\begin{equation}\label{12}
h(\uu_q')=\lim_{n\to\infty}\frac{\log |B_n(\uu_q')|}{n}=\inf_{n\ge 1}\frac{\log |B_n(\uu_q')|}{n}
\end{equation}
when  $\uu_q'$ is a subshift, where $B_n(\uu_q')$ denotes the set of  different initial words of length $n$ occurring in the sequences $(c_i)\in \uu_q'$, and $|B_n(\uu_q')|$ means the cardinality of $B_n(\uu_q')$.
(Unless otherwise stated, in this paper we use base two logarithms.)

We will complete and improve Theorems \ref{t11} and \ref{t12} in  Theorems \ref{t13}, \ref{t14} and \ref{t17} below.

\begin{theorem}\label{t13}\mbox{}
The formula \eqref{11} is valid for \emph{all} $q>1$.
\end{theorem}

We recall from \cite{DeVriesKomornik2009} that $\uu_q'$ is not always a subshift.
Theorem \ref{t13} states in particular that the limit in \eqref{12} exists even if $\uu_q'$ is not a subshift, and it is equal to the infimum in \eqref{12}.

\begin{theorem}\label{t14}\mbox{}
The function $D$ is continuous, and has a bounded variation.
\end{theorem}

Theorem \ref{t14} implies again that $D$ is differentiable almost everywhere. 
In order to describe its derivative first we establish some results on general $\beta$-expansions and on univoque bases.

Following Rényi \cite{Renyi1957} we denote by $\beta(q)=(\beta_i(q))$ the lexicographically largest expansion of $x=1$ in base $q$.
It is also called the greedy or $\beta$-expansion of $x=1$ in base $q$. 

\begin{theorem}\label{t15}
Fix  $1<r\le M+1$ arbitrarily. 
For almost all $q\in (1,r)$ there exist arbitrarily large integers $m$ such that $\beta_1(q)\cdots\beta_m(q)$ ends with more than $\log_rm$ consecutive zero digits.
\end{theorem}

This theorem improves and generalizes \cite[Theorem 2]{ErdosJooKomornik1990} concerning the case $M=1$. 
In particular, our result implies that $\beta(q)$ contains arbitrarily large blocks of consecutive zeros for almost all $q\in(1,M+1]$. 
This was first established by Erdős and Joó \cite{ErdosJoo1991} for $M=1$, and their result was extended by Schmeling \cite{Schmeling1997} for all $M$. 

Next we denote by $\uu$ the set of bases $q>1$ in which $x=1$ has a unique expansion, and by $\uuu$ its closure. 
The elements of $\uu$ are usually called \emph{univoque bases}. 

\begin{theorem}\label{t16}\mbox{}

\begin{enumerate}[\upshape (i)]
\item $\uu$ and $\uuu$ are (Lebesgue) null sets.
\item $\uu$ and $\uuu$  have Hausdorff dimension one.
\end{enumerate}
\end{theorem}

Parts (i) and (ii) were proved for $\uu$ in case $M=1$ by Erdős and Joó \cite{ErdosJoo1991} and by Daróczy and Kátai \cite{DaroczyKatai1993}, respectively.
The case of $\uuu$ hence follows because the set $\uuu\setminus\uu$ is countable (see \cite{KomornikLoreti2007}). 
Our proof of (ii) is shorter than the original one even for $M=1$.

Finally, combining Theorems  \ref{t11}, \ref{t13}, \ref{t14}, \ref{t16} (i) and  some topological results of \cite{DeVriesKomornik2009} we prove that the dimension function is a  natural variant of \emph{Devil's staircase}:

\begin{theorem}\label{t17}
\mbox{}
\begin{enumerate}[\upshape (i)]
\item $D$ is continuous in $[q',\infty)$.
\item $D'<0$ almost everywhere in $(q',\infty)$.
\item $D(q')<D(q)$ for all $q>q'$.
\end{enumerate}
\end{theorem}

\begin{remark}\label{r12}\mbox{}
Compared to the classical Cantor--Lebesgue function, we have even $D'<0$ instead of $D'=0$ almost everywhere.
\end{remark}

The paper is organized as follows.
In Section \ref{s2} we investigate the topological entropy of various subshifts that we need in the sequel.
In Section \ref{s3} we prove Theorem \ref{t13} and we prepare the proof of Theorem \ref{t14}.
Theorem \ref{t14} is proved in Section \ref{s4}, Theorems \ref{t15}--\ref{t16} in Sections \ref{s5}--\ref{s6}, and Theorem \ref{t17} in Section \ref{s7}.
Sections \ref{s5}--\ref{s6} are independent of each other and of the other sections of the paper.

\section{Topological entropies}\label{s2}

We begin by proving that the topological entropy of $\uu_q'$ is well defined even if $\uu_q'$ is not a subshift:

\begin{lemma}\label{l21}
The limit
\begin{equation*}
h(\uu_q'):=\lim_{n\to\infty}\frac{\log |B_n(\uu_q')|}{n}
\end{equation*}
exists for each $q>1$, and is equal to
\begin{equation*}
\inf_{n\ge 1}\frac{\log |B_n(\uu_q')|}{n}.
\end{equation*}
\end{lemma}

\begin{proof}
It suffices to show that the function $n\mapsto |B_n(\uu_q')|$ is submultiplicative, i.e.,
\begin{equation*}
|B_{m+n}(\uu_q')|\le |B_m(\uu_q')|\cdot |B_n(\uu_q')|
\end{equation*}
for all $m,n\ge 1$.

Denoting by $B_{k,\ell}(\uu_q')$ the set of words $c_k\cdots c_{\ell}$ where $(c_i)$ runs over $\uu_q'$, we have clearly
\begin{equation*}
|B_{m+n}(\uu_q')|=|B_{1,m+n}(\uu_q')|\le |B_{1,m}(\uu_q')|\cdot|B_{m+1,m+n}(\uu_q')|.
\end{equation*}
Notice that $|B_{m+1,m+n}(\uu_q')|\le |B_n(\uu_q')|$ because $(c_{m+i})\in\uu_q'$ for every $(c_i)\in\uu_q'$.
This completes the proof.
\end{proof}

\begin{lemma}\label{l22}\mbox{}

\begin{enumerate}[\upshape (i)]
 \item If $q\ge M+1$, then $h\left(\uu_q'\right)=\log (M+1)$.
 \item If $1<q<q'$, then $h\left(\uu_q'\right)=0$.
\end{enumerate}
\end{lemma}

\begin{proof}
If $q>M+1$, then $\uu_q'=\set{0,\ldots,M}^{\infty}$ is the full shift.
Therefore
\begin{equation*}
h\left(\uu_q'\right)=\lim_{n\to\infty}\frac{\log \abs{B_n(\uu_q')}}{n}=\lim_{n\to\infty}\frac{\log (M+1)^n}{n}=\log (M+1).
\end{equation*}

If $q=M+1$, then the above equalities remain valid. Indeed, we still have $B_n(\uu_q')=\set{0,\ldots,M}^n$ for all $n\ge 1$ because $c_1\cdots c_n(0M)^{\infty}\in\uu_q'$ for every word $c_1\cdots c_n\in\set{0,\ldots,M}^n$.

The case $1<q<q'$ follows from Theorem \ref{t12} because $\uu_q'$ is countable by \cite{GlendinningSidorov2001} (for $M=1$) and \cite{DeVriesKomornik2009}, \cite{KongLiDekking2010}, \cite{KongLi2014} (for all $M\ge 1$) and therefore $D(q)=0$.
\end{proof}

\emph{Henceforth we assume that $q'\le q\le M+1$.}
Then $x=1$ has an expansion.

We start by recalling some properties of the greedy and quasi-greedy expansions.
We denote by $\beta(q)=(\beta_i(q))$ the \emph{greedy}, i.e., the lexicographically largest expansion of $x=1$ in base $q$.
Furthermore, we denote by $\alpha(q)=(\alpha_i(q))$ the \emph{quasi-greedy}, i.e., the lexicographically largest \emph{infinite} expansion of $x=1$ in base $q$.
Here and in the sequel an expansion is called \emph{infinite} if it contains infinitely many non-zero digits.

Greedy expansions were introduced by Rényi \cite{Renyi1957}, and they were characterized by Parry \cite{Parry1960}.
Quasi-greedy expansions were introduced by Daróczy and Kátai \cite{DaroczyKatai1993}, \cite{DaroczyKatai1995}, in order to give an elegant Parry type characterization of unique expansions:

\begin{lemma}\label{l23}
A sequence $(c_i)$ belongs to $\uu_q'$ if and only if the following two conditions are satisfied:
\begin{align*}
&(c_{n+i})<\alpha(q)\quad\text{whenever}\quad c_1\ldots c_n\ne M^n, \\
&\overline{(c_{n+i})}<\alpha(q)\quad\text{whenever}\quad c_1\ldots c_n\ne 0^n.
\end{align*}
\end{lemma}
Here for a sequence $c=(c_i)$ we denote by $\overline{c}=(M-c_i)$, and  for a word $c_1\cdots c_k$ we write $\overline{c_1\cdots c_k}=(M-c_1)\cdots(M-c_k)$.

We also recall some results on the relationship between greedy and quasi-greedy expansions, and on their continuity properties:

\begin{lemma}\label{l24}\mbox{}

\begin{enumerate}[\upshape (i)]
\item If $\beta(q)$ is infinite, then $\alpha(q)=\beta(q)$.
Otherwise, $\beta(q)$ has a last non-zero digit $\beta_m(q)$, and $\alpha(q)$ is periodic with the period $\beta_1(q)\cdots \beta_{m-1}(q)(\beta_m(q)-1)$.
\item If $q_n\nearrow q$, then $\alpha(q_n)\to \alpha(q)$ component-wise.
\item If $q_n\searrow q$, then $\beta(q_n)\to \beta(q)$ component-wise.
\end{enumerate}
\end{lemma}
See, e.g., \cite{BaiocchiKomornik2007}, \cite{DeVriesKomornik2009} and \cite{DeVriesKomornik2010} for proofs.

Instead of $\uu_q'$ and $\uu_q$ it will be easier to consider the slightly modified sets
\begin{align*}
&\widetilde\uu_q':=\set{(c_i)\ :\ \overline{\alpha(q)}<(c_{m+i})<\alpha(q)\quad\text{for all}\quad m=0,1,\ldots}
\intertext{and}
&\widetilde\uu_q:=\pi_q(\widetilde\uu_q')=\set{\sum_{i=1}^{\infty}\frac{c_i}{q^i}\ :\ (c_i)\in\widetilde\uu_q'}.
\end{align*}

\begin{lemma}\label{l25}\mbox{}

\begin{enumerate}[\upshape (i)]
\item $\uu_q$ is the union of $0$, $M/(q-1)$, and of countably many sets, each similar to $\widetilde\uu_q$.
\item $\uu_q'$ and $\widetilde\uu_q'$ have the same topological entropy.
\end{enumerate}
\end{lemma}

\begin{proof}
(i) Let $(c_i)\in\uu_q'$ be different from  $0^{\infty}$ and  $M^{\infty}$.
If $0<c_1<M$, then $(c_{1+i})\in \widetilde\uu_q'$ by Lemma \ref{l23}.

If $c_1=0$, then there exists a smallest $m>1$ such that $c_m>0$, and $(c_{m+i})\in \widetilde\uu_q'$ by Lemma \ref{l23}.

If $c_1=M$, then there exists a smallest $m>1$ such that $c_m<M$, and $(c_{m+i})\in \widetilde\uu_q'$ by Lemma \ref{l23}.

It follows that $\uu_q$ is the union of $0$, $M/(q-1)$, and of the sets
\begin{align*}
&\frac{c_1}{q}+\frac{1}{q}\widetilde\uu_q,\quad c_1=1,\ldots, M-1, \\
&\frac{c_m}{q^m}+\frac{1}{q^m}\widetilde\uu_q,\quad m=2,3,\ldots,\quad c_m=1,\ldots, M, \\
&\left( \sum_{i=1}^{m-1}\frac{M}{q^i}\right) +\frac{c_m}{q^m}+\frac{1}{q^m}\widetilde\uu_q,\quad m=2,3,\ldots,\quad c_m=0,\ldots, M-1.
\end{align*}
We conclude by observing that all these sets are similar to $\widetilde\uu_q$.
\medskip

(ii) The above reasoning shows also that each word of $B_n(\uu_q')$ has the form $0^kM^{m-k}w$ or $M^k0^{m-k}w$ with some word $w\in B_{n-m}(\widetilde\uu_q')$ and some integers $k,m$ satisfying $0\le k\le m\le n$.
Hence
\begin{equation*}
\abs{B_n(\uu_q')}\le \sum_{m=0}^n2(m+1)\abs{B_{n-m}(\widetilde\uu_q')}\le (n+1)(2n+2)\abs{B_n(\widetilde\uu_q')}.
\end{equation*}
Since $\widetilde\uu_q'\subseteq \uu_q'$, it follows that
\begin{align*}
\lim_{n\to\infty}\frac{\log \abs{B_n(\widetilde\uu_q')}}{n}
&\le \lim_{n\to\infty}\frac{\log \abs{B_n(\uu_q')}}{n}\\
&\le \lim_{n\to\infty}\frac{\log (2n+2)^2\abs{B_n(\widetilde\uu_q')}}{n}\\
&=\lim_{n\to\infty}\frac{\log \abs{B_n(\widetilde\uu_q')}}{n}+\lim_{n\to\infty}\frac{2\log (2n+2)}{n}\\
&=\lim_{n\to\infty}\frac{\log \abs{B_n(\widetilde\uu_q')}}{n},
\end{align*}
whence $h(\widetilde\uu_q')=h(\uu_q')$.
\end{proof}

Since $\widetilde\uu_q'$ is not always a subshift, we introduce also the related sets
\begin{align*}
&\widetilde\vv_q':=\set{(c_i)\ :\ \overline{\alpha(q)}\le (c_{m+i})\le \alpha(q)\quad\text{for all}\quad m=0,1,\ldots}
\intertext{and}
&\widetilde\vv_q:=\pi_q(\widetilde\vv_q')=\set{\sum_{i=1}^{\infty}\frac{c_i}{q^i}\ :\ (c_i)\in\widetilde\vv_q'}.
\end{align*}

\begin{lemma}\label{l26}
$\widetilde\vv_q'$ is a subshift, and $\widetilde\uu_q'\subseteq \widetilde\vv_q'$.
\end{lemma}

\begin{proof}
If $q= M+1$, then $\alpha(q)=M^\infty$, so that $\widetilde\vv_q'=\set{0,1,\cdots,M}^\infty$ is the full shift.

Henceforth assume that $q<M+1$, and consider the set $F$ of all finite blocks $d_1\cdots d_n\in\set{0,\ldots,M}^n$ (of arbitrary length), satisfying one of the lexicographic inequalities
\begin{equation*}
d_1\cdots d_n<\overline{\alpha_1(q)\cdots\alpha_n(q)}\quad\text{and}\quad d_1\cdots d_n>\alpha_1(q)\cdots\alpha_n(q).
\end{equation*}
By definition, none of these blocks appear in any $(c_i)\in \widetilde\vv_q'$.

Conversely, if $(c_i)\in\set{0,1,\cdots,M}^\infty\setminus\widetilde\vv_q'$, then there is a positive integer $m$ such that either
\begin{equation*}
c_mc_{m+1}\cdots<\overline{\alpha(q)}
\end{equation*}
or
\begin{equation*}
c_{m}c_{m+1}\cdots>\alpha(q),
\end{equation*}
and hence there is another positive integer $n$ such that either
\begin{equation*}
c_m\cdots c_{m+n}<\overline{\alpha_1(q)\cdots\alpha_n(q)}
\end{equation*}
or
\begin{equation*}
c_m\cdots c_{m+n}>\alpha_1(q)\cdots\alpha_n(q).
\end{equation*}
Hence $(c_i)$ contains at least one block from $F$.

The inclusion $\widetilde\uu_q'\subseteq\widetilde\vv_q'$ is obvious from the definition.
\end{proof}

Since $\widetilde\uu_q'$ is not always a subshift of finite type, we introduce for each positive integer $n$ the set $\widetilde\uu_{q,n}'$ of sequence $(c_i)$ satisfying for all $m=0,1,\ldots$ the inequalities
\begin{equation*}
\overline{\alpha_1(q)\cdots\alpha_n(q)}<c_{m+1}\cdots c_{m+n}<\alpha_1(q)\cdots\alpha_n(q).
\end{equation*}
Similarly, we define the sets $\widetilde\vv_{q,n}'$ and $\widetilde\ww_{q,n}'$ by replacing the above inequalities by
\begin{equation*}
\overline{\alpha_1(q)\cdots\alpha_n(q)}\le c_{m+1}\cdots c_{m+n}\le\alpha_1(q)\cdots\alpha_n(q)
\end{equation*}
and
\begin{equation*}
\overline{\beta_1(q)\cdots\beta_n(q)}\le c_{m+1}\cdots c_{m+n}\le\beta_1(q)\cdots\beta_n(q),
\end{equation*}
respectively.

\begin{lemma}\label{l27}
$\widetilde\uu_{q,n}'$, $\widetilde\vv_{q,n}'$ and $\widetilde\ww_{q,n}'$ are subshifts of finite type, and
\begin{equation}\label{21}
\widetilde\uu_{q,n}'\subseteq\widetilde\uu_q'\subseteq\widetilde\vv_q'\subseteq\widetilde\vv_{q,n}'\subseteq\widetilde\ww_{q,n}'
\end{equation}
for all $n$.

Furthermore, the sets $\widetilde\uu_{q,n}'$ are increasing, while $\widetilde\vv_{q,n}'$ and $\widetilde\ww_{q,n}'$ are decreasing when $n$ is increasing.
\end{lemma}

\begin{proof}
It is clear that $\widetilde\uu_{q,n}'$ is characterized by the finite set of forbidden blocks $d_1\cdots d_n\in\set{0,\ldots,M}^n$ satisfying the lexicographic inequalities
\begin{equation*}
d_1\cdots d_n\le\overline{\alpha_1(q)\cdots\alpha_n(q)}\quad\text{or}\quad d_1\cdots d_n\ge \alpha_1(q)\cdots\alpha_n(q).
\end{equation*}
Hence it is a subshift of finite type.

The proof for $\widetilde\vv_{q,n}'$ and $\widetilde\ww_{q,n}'$ is analogous.

The remaining assertions follow from the definition of lexicographic inequalities.
\end{proof}

We are going to show  that these sets well approximate $\widetilde\uu_q'$:

\begin{proposition}\label{p28}
For $q\in[q', M+1]$ we have
\begin{equation*}
\lim_{n\to\infty}h(\widetilde\uu_{q,n}')=\lim_{n\to\infty}h(\widetilde\vv_{q,n}')=\lim_{n\to\infty}h(\widetilde\ww_{q,n}')=h(\widetilde\uu_q')=h(\widetilde\vv_q').
\end{equation*}
\end{proposition}

The proof of the proposition is divided into a series of lemmas.

\begin{lemma}\label{l29}
Let $q'\le q<p\le M+1$. Then
\begin{equation*}
\widetilde\ww_{q,n}'\subseteq\widetilde\uu'_{p,n}
\end{equation*}
for all sufficiently large $n$.
\end{lemma}

\begin{proof}
Since there are only countably many finite greedy expansions, the set
\begin{equation*}
\set{r\in(1,M+1]\ :\ \beta(r)\ne\alpha(r)}
\end{equation*}
is countable.
There exists therefore $r\in (q,p)$ such that $\beta(r)=\alpha(r)$, and then
\begin{equation*}
\beta(q)<\beta(r)=\alpha(r)<\alpha(p)
\end{equation*}
because the maps $r\mapsto\beta(r)$ and $r\mapsto\alpha(r)$ are strictly increasing by the definition of the greedy and quasi-greedy algorithms.

Fix a sufficiently large $n$ such that
\begin{equation*}
\alpha_1(p)\cdots\alpha_n(p)>\beta_1(q)\cdots\beta_n(q).
\end{equation*}
If $d=(d_i)\in\widetilde\ww_{q,n}'$, then
\begin{equation*}
d_{m+1}\cdots d_{m+n}\le \beta_1(q)\cdots\beta_n(q)<\alpha_1(p)\cdots\alpha_n(p)
\end{equation*}
and symmetrically
\begin{equation*}
d_{m+1}\cdots d_{m+n}\ge \overline{\beta_1(q)\cdots\beta_n(q)}> \overline{\alpha_1(p)\cdots\alpha_n(p)}
\end{equation*}
for all $m\geq 0$, i.e., $d\in\widetilde\uu_{p,n}'$.
\end{proof}

We  recall $\uu$ is the set of bases $q>1$ in which $x=1$ has a unique expansion, and  $\overline{\uu}$ is its closure.
Furthermore, we recall from \cite{KomornikLoreti2007} that $q\in\overline{\uu}$ if and only if
\begin{equation}\label{22}
\overline{\alpha_1(q)\alpha_2(q)\cdots}<\alpha_{k+1}(q)\alpha_{k+2}(q)\cdots\le \alpha_1(q)\alpha_2(q)\cdots
\end{equation}
for all $k\ge 0$.
Moreover, there exists infinitely many indices $n$ such that
\begin{equation}\label{23}
\overline{\alpha_1(q)\cdots\alpha_{n-k}(q)}<\alpha_{k+1}(q)\cdots\alpha_n(q)\le\alpha_1(q)\cdots\alpha_{n-k}(q)
\end{equation}
for all $0\le k\le n-1$.
In particular, $\alpha_n(q)>0$ for these indices.

\begin{lemma}\label{l210}
Let $q\in\overline{\uu}$ and $(\alpha_i)=\alpha(q)$.

\begin{enumerate}[\upshape (i)]
\item For each $n\ge 1$,  $B_n(\widetilde\vv_q')=B_n(\widetilde\vv_{q,n}')$ is the set of words $d_1\cdots d_n$ satisfying
\begin{equation}\label{24}
\overline{\alpha_1\cdots\alpha_{n-k}}\le d_{k+1}\cdots d_n\le \alpha_1\cdots\alpha_{n-k}
\end{equation}
for all $0\le k\le n-1$.
\item For each $n\ge 1$ satisfies \eqref{23},   $B_n(\widetilde\uu_{q,n}')$ is the set of words $d_1\cdots d_n$ satisfying
\begin{equation}\label{25}
\overline{\alpha_1\cdots\alpha_n}<d_1\cdots d_n<\alpha_1\cdots\alpha_n,
\end{equation}
and relations \eqref{24} for all $1\le k\le n-1$.
\item If $n\ge 1$ satisfying \eqref{23}, then
\begin{equation*}
B_n(\widetilde\vv_{q,n}')\setminus B_n(\widetilde\uu_{q,n}')=\set{\alpha _1(q)\dots \alpha _n(q), \overline{\alpha _1(q)\dots \alpha _n(q)}}.
\end{equation*}
\end{enumerate}
\end{lemma}
\remove{We remark that Part (i) and its proof remain valid for $q\in\vv$.}

\begin{proof}
(i) Note that $B_n(\widetilde\vv_{q}')\subseteq B_n(\widetilde\vv_{q,n}')$, and that each word of $B_n(\widetilde\vv_{q,n}')$ satisfies the relations \eqref{24}.
It remains to prove that if a word $d_1\cdots d_n$ satisfies the relations \eqref{24}  for all $0\le k\le n-1$, then it belongs to $B_n(\widetilde\vv_q')$.

Let $0\le k_1\le n$ be the first integer such that either
\begin{equation*}
d_{k_1+1}\cdots d_n=\alpha_1\cdots\alpha_{n-k_1}
\end{equation*}
or
\begin{equation*}
d_{k_1+1}\cdots d_n= \overline{\alpha_1\cdots\alpha_{n-k_1}}.
\end{equation*}
Assume by symmetry that
\begin{equation}\label{26}
d_{k_1+1}\cdots d_n=\alpha_1\cdots\alpha_{n-k_1}
\end{equation}

The minimality of $k_1$ implies that
\begin{equation*}
\overline{\alpha_1\cdots\alpha_{n-k}}< d_{k+1}\cdots d_n<\alpha_1\cdots\alpha_{n-k} \quad\textrm{for any}~0\le k< k_1.
\end{equation*}
Combining this with \eqref{22} we conclude that
\begin{equation*}
d_1\cdots d_n\alpha_{n-k_1+1}\alpha_{n-k_1+2}\cdots= d_1\cdots d_{k_1}\alpha_{1}\alpha_2\cdots\in\widetilde\vv_q';
\end{equation*}
hence $d_1\cdots d_n\in B_n(\widetilde\vv_q')$.
\medskip

(ii) Take $n$ satisfying \eqref{23}, and note that each word of $B_n(\widetilde\uu_{q,n}')$ satisfies the above mentioned relations.
It remains to prove that if a word $d_1\cdots d_n$ satisfying \eqref{25}, and relations \eqref{24} for all $1\le k\le n-1$, then it belongs to $B_n(\widetilde\uu_{q,n}')$.

Choosing $k_1$ as in (i), now we have $k_1\ge 1$.
We may assume \eqref{26} again.
Using \eqref{23} it follows that $\alpha_n>0$ and
\begin{equation*}
\alpha_{k+1}\cdots\alpha_{n-1}\alpha_n^{-}\ge \overline{\alpha_1\cdots \alpha_{n-k}}\quad\text{and}\quad\alpha_1\cdots\alpha_k> \overline{\alpha_{n-k+1}\cdots \alpha_n}
\end{equation*}
for all $0\le k<n$, where we write $\alpha_n^-:=\alpha_n-1$.
Hence,
\begin{multline*}
d_1\cdots d_{n}(\alpha_{n-k_1+1}\cdots \alpha_{n-1}\alpha_{n}^-\alpha_1\cdots\alpha_{n-k_1})^\infty\\
=d_1\cdots d_{k_1}(\alpha_1\cdots\alpha_{n-1}\alpha_n^-)^\infty\in\widetilde\uu_{q,n}',
\end{multline*}
and therefore $d_1\cdots d_n\in B_n(\widetilde\uu_{q,n}')$.
\medskip

(iii) This follows from (i), (ii) and \eqref{22}.
\end{proof}

We also need the following lemma, where we use the set $\uuu$ defined in the introduction.

\begin{lemma}\label{l211}
If $p$ and $q$ belong to the same connected component of $(1,\infty)\setminus\uuu$, then $h(\uu_p')=h(\uu_q')$ and $h(\widetilde\uu_p')=h(\widetilde\uu_q')$.
\end{lemma}

\begin{proof}
By Lemma \ref{l25} (ii) it suffices to prove the equalities $h(\uu_p')=h(\uu_q')$. 

Consider an arbitrary connected component $I=(q_0,q_0^*)$. 
We recall from \cite[Theorem 1.7]{DeVriesKomornik2009} that there exists a sequence $(q_n)$ satisfying $q_0<q_1<\cdots$ and converging to $q_0^*$, and such that
\begin{equation*}
\uu_q'=\uu_{q_n}'\quad\text{for all}\quad q\in (q_{n-1},q_n),\quad n=1,2,\ldots .
\end{equation*}

The remaining equalities $h\left( \uu_{q_n}'\right) =h\left( \uu_{q_{n+1}}'\right)$ were shown during the proof of \cite[Theorem 2.6]{KongLi2014}. 
\end{proof}

Finally we recall the Perron--Frobenius Theorem (see \cite[Theorem 4.4.4]{LindMarcus1995}):

\begin{lemma}\label{l212}
Let  $G(n)$ be an edge graph representation of $\widetilde\uu_{q,n}'$, and   $\lambda_n$  its spectral radius.
Then there exist  positive constants $c_1, c_2$ such that
\begin{equation*}
c_1 \lambda_n^k\le  |B_k(\widetilde\uu_{q,n}')|\le c_2 k^s\lambda_n^k
\end{equation*}
for all $k\ge 1$, where $s$ denotes the number of strongly connected components of $G(n)$.

If $G(n)$ is strongly connected, then the factor $k^s$ may be omitted in the second inequality.
\end{lemma}

\begin{proof}[Proof of Proposition \ref{p28}]
All indicated topological entropies are well defined by Lemmas \ref{l26} and \ref{l27}.
Furthermore, the monotonicity of the set sequences $(\widetilde\uu_{q,n}')$, $(\widetilde\vv_{q,n}')$ and $(\widetilde\ww_{q,n}')$ implies the existence of the indicated limits as $n\to\infty$.

If $q\in[q',M+1]\setminus\overline{\uu}$, then $q\in (q', M+1)$ (because $q', M+1\in\uu$).
Applying Lemma \ref{l211} we may choose a neighbourhood $(q_1, q_2)$ of $q$ such that  $h(\widetilde\uu_p')=h(\widetilde\uu_q')$ for all $p\in[q_1, q_2]$.
Using Lemmas \ref{l27} and \ref{l29} we obtain that
\begin{equation*}
\widetilde\uu_{q_1}'\subseteq\widetilde\uu_{q,n}'\subseteq\widetilde\ww_{q,n}'\subseteq\widetilde\uu_{q_2}'
\end{equation*}
for all sufficiently large indices $n$, and therefore
\begin{equation*}
\lim_{n\rightarrow\infty}h(\widetilde\uu_{q,n}')=\lim_{n\rightarrow\infty}h(\widetilde\ww_{q,n}')=h(\widetilde\uu_q').
\end{equation*}

Henceforth we assume that $q\in \overline{\uu}$.
In view of the inclusions \eqref{21} it is sufficient to prove  that
\begin{equation}\label{27}
\lim_{n\to\infty}h(\widetilde\ww_{q,n}')\le h(\widetilde\vv_q')
\end{equation}
and
\begin{equation}\label{28}
\lim_{n\to\infty}h(\widetilde\vv_{q,n}')\le \lim_{n\to\infty}h(\widetilde\uu_{q,n}').
\end{equation}

First we show that
\begin{equation*}
|B_n(\widetilde\ww_{q,n}')|\le 2(n+1)^2|B_n(\widetilde\vv_q')|
\end{equation*}
for all $n\ge 1$.
If $\alpha(q)=\beta(q)$, then $\widetilde\ww_{q,n}'=\widetilde\vv_{q,n}'$ and therefore by Lemma \ref{l210} we have  $B_n(\widetilde\ww_{q,n}')=B_n(\widetilde\vv_q')$ for all $n$.

If $\alpha(q)\ne\beta(q)$, then $\beta(q)$ has a last nonzero digit $\beta_m$, and by Lemma \ref{l24} $\alpha(q)$ is periodic with the period  $\beta_1(q)\cdots \beta_{m-1}(q)\beta^-_m(q)$.
In this case, if $d_1\cdots d_n\in B_n(\widetilde\ww_{q,n}')\setminus B_n(\widetilde\vv_{q}')$, then {for} any $0\le k\le n-1$
\begin{equation*}
\overline{\beta_1(q)\cdots\beta_{n-k}(q)}\le d_{k+1}\cdots d_n\le \beta_1(q)\cdots \beta_{n-k}(q),
\end{equation*}
and by Lemma \ref{l210}  it follows that there exists a least integer $0\le k\le n-1$ such that
either
\begin{equation*}
d_{k+1}\cdots d_n<{\overline{\alpha_1(q)\cdots\alpha_{n-k}(q)}}
\end{equation*}
or
\begin{equation*}
d_{k+1}\cdots d_n>\alpha_1(q)\cdots\alpha_{n-k}(q).
\end{equation*}
This implies that $d_{k+1}\cdots d_n$ or $\overline{d_{k+1}\cdots d_n}$ must be of the form
\begin{equation*}
(\alpha_1(q)\cdots \alpha_m(q))^j\beta_1(q)\cdots\beta_{n-k-mj}(q),\quad j=0,1,\cdots,[(n-k)/m].
\end{equation*}
The number of these words can not exceed $2(n+1)$.
Moreover, by the minimality of $k$ and Lemmas \ref{l27}, \ref{l210} it follows that
\begin{equation*}
d_1\cdots d_k\in B_k(\widetilde\vv_{q,k}')=B_k(\widetilde\vv_{q}')=B_k(\widetilde\vv_{q,n}').
\end{equation*}
Hence
\begin{multline*}
|B_n(\widetilde\ww_{q,n}')|-|B_n(\widetilde\vv_{q,n}')|=|B_n(\widetilde\ww_{q,n}')\setminus B_n(\widetilde\vv_{q,n}')|\\
\le 2(n+1)\sum_{k=0}^{n-1}|B_{k}(\widetilde\vv_{q,n}')|\le 2n(n+1) |B_n(\widetilde\vv_{q,n}')|,
\end{multline*}
and the required estimate follows.

Using this estimate we have
\begin{multline*}
h(\widetilde\ww_{q,n}')
=\inf_{k\ge 1}\frac{\log |B_k(\widetilde\ww_{q,n}')|}{k}\le \frac{\log| B_n(\widetilde\ww_{q,n}')|}{n}\\
\le \frac{\log |B_n(\widetilde\vv_q')|+\log2+2\log(n+1)}{n}.
\end{multline*}
Letting $n\to\infty$ the relation \eqref{27} follows.

Turning to the proof of the relation \eqref{28},  first we consider the case $q=q'$.
Using \eqref{21} and \eqref{27} it follows that
\begin{equation*}
\lim_{n\rightarrow\infty}h(\widetilde\vv_{q,n}')=h(\widetilde\vv_q').
\end{equation*}

Furthermore, we also deduce from \eqref{21} and Lemma \ref{l210} that
\begin{equation*}
|B_{n_k}(\widetilde\vv_q')\setminus B_{n_k}(\widetilde\uu_q')|\le |B_{n_k}(\widetilde\vv_{q,n_k})\setminus B_{n_k}(\widetilde\uu_{q,n_k}')|=2,
\end{equation*}
where $(n_k)$ is a sequence of indices satisfying \eqref{23}.
Hence
\begin{equation*}
\begin{split}
h(\widetilde\vv_q')
&=\lim_{n\rightarrow\infty}\frac{\log |B_n(\widetilde\vv_q')|}{n}=\lim_{k\rightarrow\infty}\frac{\log |B_{n_k}(\widetilde\vv_q')|}{n_k}\\
&\le \lim_{k\rightarrow\infty}\frac{\log \big(|B_{n_k}(\widetilde\uu_q')|+2\big)}{n_k}=h(\widetilde\uu_q').
\end{split}
\end{equation*}
The existence of the last limit and the last equality follows from Lemma \ref{l21}.

Since $h(\widetilde\uu_q')=0$ for $q= q'$ by Theorem \ref{t11}, we conclude that
\begin{equation*}
\lim_{n\rightarrow\infty}h(\widetilde\vv_{q,n}')=0.
\end{equation*}

Assume henceforth that $q>q'$, so that $h(\widetilde\uu_q')>0$. 
This was proved in \cite{GlendinningSidorov2001} for $M=1$, and the proof remains valid for all odd values of $M$, and in \cite[Lemma 4.10]{KongLiDekking2010} for $M=2,4,\ldots .$
For each $n\ge N$ we have $h(\widetilde\uu_{q,n}')=\log\lambda_n$ with the notations of Lemma \ref{l212}, and
\begin{equation*}
\lambda_n\ge \lambda_N>1
\end{equation*}
by the increasingness of the set sequence $(\widetilde\uu_{q,n}')$.
We are going to estimate the size of  $B_k(\widetilde\vv_{q,n}')\setminus B_k(\widetilde\uu_{q,n}')$ for each fixed $n\ge N$ satisfying \eqref{23}  and $k\ge n$.

Let us denote by $G'(n)$  the edge graph representing $\widetilde\vv_{q,n}'$, and set $u=\alpha_1(q)\cdots\alpha_n(q)$.
Then $G(n)$ is a subgraph of $G'(n)$, and the words $u$ and $\overline{u}$ are forbidden in $G(n)$.
We seek an upper bound for $|B_k(\widetilde\vv_{q,n}')\setminus B_k(\widetilde\uu_{q,n}')|$.

Suppose that $d_1\cdots d_k\in B_k(\widetilde\vv_{q,n}')\setminus B_k(\widetilde\uu_{q,n}')$. Then by Lemma \ref{l210} it follows that
the word $d_1\cdots d_k$ must contain at least once  $u$ or $\overline{u}$.
If it contains exactly $r\ge 1$ times $u$ or $\overline{u}$, then it has the form
\begin{equation*}
d_1\cdots d_k=\omega_0\tau_1\omega_1\cdots\tau_r\omega_r
\end{equation*}
where each $\tau_j$ is equal to $u$ or $\overline{u}$, and $k_0+\cdots+k_r=k-rn$, where $k_j\ge 0$ denotes the length of $\omega_j$.

Assuming first that the graph $G(n)$ is strongly connected, we may apply Lemma \ref{l212} without the factor $k^s$.
Assuming without loss of generality that $c_1\le 1\le c_2$, we obtain the following estimate:
\begin{align*}
|B_k(\widetilde\vv_{q,n}')|&\le |B_k(\widetilde\uu_{q,n}')|+\sum_{r=1}^{[k/n]}\sum_{k_0+\cdots+k_r=k-nr}2^r\prod_{j=0}^r(c_2\lambda_n^{k_j})\\
&= |B_k(\widetilde\uu_{q,n}')|+c_2\lambda_n^k\sum_{r=1}^{[k/n]}\sum_{k_0+\cdots+k_r=k-nr}(2c_2\lambda_n^{-n})^r\\
&=|B_k(\widetilde\uu_{q,n}')|+c_2\lambda_n^k\sum_{r=1}^{[k/n]}\binom{k-r(n-1)}{r}(2c_2\lambda_n^{-n})^r\\
&\le |B_k(\widetilde\uu_{q,n}')|+c_2\lambda_n^k\sum_{r=1}^k\binom{k}{r}(2c_2\lambda_n^{-n})^r\\
&\le |B_k(\widetilde\uu_{q,n}')|\frac{c_2}{c_1} \sum_{r=0}^k\binom{k}{r}(2c_2\lambda_n^{-n})^r\\
&=|B_k(\widetilde\uu_{q,n}')|\frac{c_2}{c_1}(1+2c_2\lambda_n^{-n})^k\\
&\le |B_k(\widetilde\uu_{q,n}')|\frac{c_2}{c_1}(1+2c_2\lambda_N^{-n})^k.
\end{align*}

If the graph $G(n)$ is not strongly connected, then we distinguish two cases:
\begin{itemize}
\item If $u$ and $\overline{u}$ belong to the same strongly connected component of $G'(n)$, then we have to change $c_2\lambda_n^{k_j}$  to $c_2k_j^s\lambda_n^{k_j}$ in the above estimate for $j=0$ and $j=r$.
\item If $u$ and $\overline{u}$ belong to different strongly connected components of $G'(n)$, then for each $d_1\cdots d_k$ there is an index $0\le r'\le r$ such that either $\tau_j=u\Longleftrightarrow j\le r'$ or $\tau_j=u\Longleftrightarrow j>r'$.
Then we may change the above factor $2^r$ to $r+1$, and we have to change $c_2\lambda_n^{k_j}$ to $c_2k_j^s\lambda_n^{k_j}$ for $j=0$, $j=r'$ and $j=r$.
\end{itemize}

Summarizing, we obtain in all cases the following estimate:
\begin{equation*}
|B_k(\widetilde\vv_{q,n}')|\le |B_k(\widetilde\uu_{q,n}')|\frac{c_2}{c_1}k^{3s}(1+2c_2\lambda_N^{-n})^k.
\end{equation*}
It follows that
\begin{multline*}
\frac{\log |B_k(\widetilde\vv_{q,n}')|}{k}\le \frac{\log |B_k(\widetilde\uu_{q,n}')|}{k}\\
+\frac{\log (c_2/c_1)}{k}+3s\frac{\log k}{k}+\log (1+2c_2\lambda_N^{-n})
\end{multline*}
for all $k\ge n$.
Letting $k\to\infty$ we conclude that
\begin{equation*}
h(\widetilde\vv_{q,n}')\le h(\widetilde\uu_{q,n}')+\log (1+2c_2\lambda_N^{-n})
\end{equation*}
for all $n\ge N$ satisfying \eqref{23}.
Since $\lambda_N>1$, taking $n$ satisfying \eqref{23} and letting $n\to\infty$ we get \eqref{28}.
\end{proof}

\section{Proof of Theorem \ref{t13}}\label{s3}

First we consider the cases $1<q<q'$ and $q\ge M+1$.

\begin{lemma}\label{l31}\mbox{}

\begin{enumerate}[\upshape (i)]
\item The formula \eqref{11} holds for $1<q<q'$ with $D(q)=h(\uu_q')=0$.
\item The formula \eqref{11} holds for all $q\ge M+1$ with  $h(\uu_q')=\log (M+1)$.
\end{enumerate}
\end{lemma}

\begin{proof}
(i) We have shown in Lemma \ref{l22} that $h(\uu_q')=0$.
Since $\uu_q$ is countable (see the proof of Lemma \ref{l22}), we have also $D(q)=0$.
\medskip

(ii) We have shown in Lemma \ref{l22} that $h(\uu_q')=\log (M+1)$.

Since $[0,1]\setminus\uu_{M+1}$ and $\set{0,\ldots,M}^{\infty}\setminus\uu_{M+1}'$ are countable, we have $D(M+1)=1$ and $h\left(\uu_{M+1}'\right)=\log (M+1)$.

If $q>M+1$, then $\uu_q'=\set{0,\ldots,M}^{\infty}$, so that $h\left(\uu_q'\right)=\log (M+1)$, and $\uu_q$ is a self-similar set satisfying the relation
\begin{equation*}
\uu_q=\bigcup_{j=0}^{M}\left( \frac{j}{q}+\frac{1}{q}\uu_q\right).
\end{equation*}
The union is disjoint because each $x\in\uu_q$ has a unique expansion.

Observe that $\uu_q$ is a non-empty \emph{compact} set.
Indeed, it is bounded because $\uu_q\subseteq  [0,M/(q-1)]$.
It remains to show that it is closed, i.e, if $(x_k)\subset\uu_q$ converges to some real number $x$, then $x\in\uu_q$.

If two expansions $(a_i)$ and $(b_i)$ first differ at the $m$th position, then
\begin{equation*}
\abs{\sum_{i=1}^{\infty}\frac{a_i}{q^i}-\sum_{i=1}^{\infty}\frac{b_i}{q^i}}\ge \frac{1}{q^m}-\sum_{i=m+1}^{\infty}\frac{M}{q^i}=\frac{q-M-1}{q^m(q-1)}>0.
\end{equation*}
Using this estimate we obtain that the expansion of $x_k$ converges com\-po\-nent-wise to some sequence $(c_i)$, and that $(c_i)$ is the (necessarily unique) expansion of $x$.

Applying \cite{Hutchinson1981} (see also \cite[Proposition 9.7]{Falconer2003}) we conclude that $r:=D(q)$ is the solution of the equation $(M+1)q^{-r}=1$, yielding
\begin{equation*}
D(q)=\frac{\log (M+1)}{\log q}.\qedhere
\end{equation*}
\end{proof}

In view of Theorem \ref{t11} and Lemma \ref{l31} it remains to investigate the dimension function
\begin{equation*}
D(q)=\dime_H\uu_q=\dim_H\widetilde\uu_q
\end{equation*}
for $q'\le q\le M+1$.

\begin{lemma}\label{l32}
Let $q\in [q',M+1)$.
There exists a positive integer $n(q)$ and a real number $\eps(q)>0$ such that
\begin{equation*}
\dim_H\pi_p(\widetilde\uu_{q,n}')=\frac{h(\widetilde\uu_{q,n}')}{\log p}
\quad\text{and}\quad
\dim_H\pi_p(\widetilde\vv_{q,n}')=\frac{h(\widetilde\vv_{q,n}')}{\log p}
\end{equation*}
for all $n\ge n(q)$ and $p\in (q-\eps(q),q]$.
\end{lemma}

\begin{proof}
The two cases being similar, we consider only that of $\widetilde\vv_{q,n}'$.

Let $N$ be the smallest index satisfying $\alpha_N(q)<M$, and fix $n>N$ such that $q^{n-N}(q-1)>M$.
Let $p\in (q',q]$ be sufficiently close to $q$ such that
\begin{equation*}
p^{n-N}(p-1)>M\quad\text{and}\quad \alpha_i(p)=\alpha_i(q),\quad i=1,\ldots, n.
\end{equation*}

We know already that $\widetilde\vv_{q,n}'$ is  a subshift of finite type corresponding to the finite set $F_n$ of forbidden blocks $d_1\cdots d_n\in\set{0,\ldots,M}^n$ satisfying one of the lexicographic inequalities
\begin{equation*}
d_1\cdots d_n<\overline{\alpha_1(q)\cdots\alpha_n(q)}\quad\text{and}\quad d_1\cdots d_n>\alpha_1(q)\cdots\alpha_n(q).
\end{equation*}
We finish the proof by  showing that $\pi_p(\widetilde\vv_{q,n}')$ is a graph-directed set satisfying the {strong separation condition}: then we may conclude by using the results of  Mauldin and Williams \cite{MauldinWilliams1988}.
We argue similarly to \cite[Lemma 6.4]{KongLi2014}.

Let us denote by $\mathbf{G}=(G,V,E)$ the edge graph with the vertex set
\begin{equation*}
V:=B_{n-1}(\widetilde\vv_{q,n}')=\set{d_1\cdots d_{n-1}\in\set{0,\ldots,M}^{n-1}\ :\ d\in\widetilde\vv_{q,n}'}.
\end{equation*}
For two vertices $\mathbf{u}=u_1\cdots u_{n-1}$ and $\mathbf{v}=v_1\cdots v_{n-1}$ we draw an edge $\mathbf{uv}\in E$ from $\mathbf{u}$ to $\mathbf{v}$ and label it $\ell_{\mathbf{uv}}=u_1$ if
\begin{equation*}
u_2\cdots u_{n-1}=v_1\cdots v_{n-2}\quad\text{and}\quad u_1\cdots u_{n-1}v_{n-1}\notin F_n.
\end{equation*}
Then the edge graph $\mathbf{G}=(G,V,E)$ is a representation of $\widetilde\vv_{q,n}'$ (see \cite{LindMarcus1995}).

For $\mathbf{u}=u_1\cdots u_{n-1}\in V$ we set
\begin{multline*}
K_{\mathbf{u}}:=\Bigl\{\sum_{i=1}^\infty\frac{d_i}{p^i}\ :\  d_i=u_i\text{ for }i=1,\ldots,n-1,\\
\text{and } d_{m+1}\cdots d_{m+n}\notin{F_n} \text{ for all }m\ge 0\Bigr\}.
\end{multline*}
For each edge $\mathbf{uv}\in{E}$ with vertices
\begin{equation*}
\mathbf{u}=u_1\cdots u_{n-1},\quad \mathbf{v}=v_1\cdots v_{n-1}
\end{equation*}
we define
\begin{equation*}
f_{\mathbf{uv}}(x):=\frac{x+\ell_{uv}}{p}=\frac{x+u_1}{p}.
\end{equation*}
Then one can verify that
\begin{equation*}
\pi_p(\widetilde\vv_{q,n}')=\bigcup_{\mathbf{u}\in V}K_{\mathbf{u}}=\bigcup_{\mathbf{u}\in V}\bigcup_{\mathbf{uv}\in {E}}f_{\mathbf{uv}}(K_\mathbf{v}),
\end{equation*}
so that $\pi_p(\widetilde\vv_{q,n}')$ is a graph-directed set (see \cite{MauldinWilliams1988}).

It remains to show that
\begin{equation*}
f_{\mathbf{uv}}(K_\mathbf{v})\cap f_{\mathbf{uv'}}(K_{\mathbf{v'}})=\emptyset
\end{equation*}
for all $\mathbf{uv}, \mathbf{uv'}\in E$ with $\mathbf{v}\ne\mathbf{v'}$.

Let $\mathbf{uv}, \mathbf{uv'}$ be two such edges in $E$ with
\begin{equation*}
\mathbf{u}=u_1\cdots u_{n-1},\quad \mathbf{v}=v_1\cdots v_{n-1}\quad\text{and}\quad \mathbf{v'}=v_1'\cdots v_{n-1}'.
\end{equation*}
Then
\begin{equation*}
v_1\cdots v_{n-2}=u_2\cdots u_{n-1}=v'_1\cdots v_{n-2}'.
\end{equation*}
Assume that $v_{n-1}<v_{n-1}'$.
Then it suffices to show that for any
\begin{equation*}
x=\pi_p(v_1\cdots v_{n-1}c_1 c_2\cdots)\in K_\mathbf{v},\quad y=\pi_p(v_1'\cdots v_{n-1}' d_1d_2\cdots)\in K_{\mathbf{v}'}
\end{equation*}
we have $f_{\mathbf{uv}}(x)<f_{\mathbf{uv}'}(y)$, i.e.,
\begin{equation*}
\sum_{i=1}^{n-1}\frac{u_i}{p^i}+\frac{v_{n-1}}{p^n}+\frac{1}{p^n}\sum_{i=1}^\infty\frac{c_i}{p^i}<\sum_{i=1}^{n-1}\frac{u_i}{p^i}+\frac{v_{n-1}'}{p^n}
+\frac{1}{p^n}\sum_{i=1}^\infty\frac{d_i}{p^i}.
\end{equation*}
This is equivalent to the inequality
\begin{equation*}
\pi_p(c)<v_{n-1}'-v_{n-1}+\pi_p(d).
\end{equation*}
This follows from our choice of $N$ and $p$ at the beginning of the proof.
Indeed, using the relations
\begin{equation*}
\alpha_{k+1}(q)\cdots\alpha_{k+N}(q)\le M^{N-1}(M-1)\quad k=0, 1,2,\ldots
\end{equation*}
we have
\begin{equation*}
\begin{split}
\pi_p(c)&\le\pi_p\big(( M^{N-1}(M-1) )^\infty\big)=\frac{M}{p-1}-\frac{1}{p^N-1}\\
&<\frac{M}{p-1}-\frac{M}{p^{n}(p-1)}=\pi_p(M^n 0^\infty)\\
&=\pi_p(\alpha_1(q)\cdots\alpha_n(q)\; 0^\infty)+\pi_p(\overline{\alpha_1(q)\cdots \alpha_n(q)}\; 0^\infty)\\
&<\pi_p(\alpha(p))+\pi_p(d)=1+\pi_p(d).\qedhere
\end{split}
\end{equation*}
\end{proof}

\begin{lemma}\label{l33}
Let $q\in[q',M+1)$.
There exists a positive integer $n(q)$ and a real number $\eps(q)>0$ such that
\begin{equation*}
\dim_H\pi_p(\widetilde\uu_{q,n}')=\frac{h(\widetilde\uu_{q,n}')}{\log p}
\quad\text{and}\quad
\dim_H\pi_p(\widetilde\ww_{q,n}')=\frac{h(\widetilde\ww_{q,n}')}{\log p}
\end{equation*}
for all $n\ge n(q)$ and $p\in [q,q+\eps(q))$.
\end{lemma}

\begin{proof}
We only give the proof for   $\widetilde\ww_{q,n}'$.

Let $N$ be the smallest index satisfying $\beta_N(q)<M$, and fix $n>N$ such that $q^{n-N}(q-1)>M$.
Let $p\in [q,M+1)$ be sufficiently close to $q$ such that
\begin{equation*}
\beta_i(p)=\beta_i(q),\quad i=1,\ldots, n.
\end{equation*}
Since $p\ge q$, we have also $p^{n-N}(p-1)>M$.

Similarly to the proof of Lemma \ref{l32} we construct an edge graph representing $\widetilde\ww_{q,n}'$, and hence $\pi_p(\widetilde\ww_{q,n}')$ is a graph-directed set.
Then it suffices to prove that the corresponding iterated function system satisfies the open set condition, i.e.,
\begin{equation*}
\pi_p(c)<1+\pi_p(d)
\end{equation*}
for all $c, d\in\widetilde\ww_{q,n}'$.

This follows again from our choice of $N$ and $p$ at the beginning of the proof.
Indeed, using the relations
\begin{equation*}
\beta_{k+1}(q)\cdots\beta_{k+N}(q)\le M^{N-1}(M-1)\quad k=0, 1,2,\ldots
\end{equation*}
we have
\begin{equation*}
\begin{split}
\pi_p(c)&\le\pi_p\big(( M^{N-1}(M-1) )^\infty\big)=\frac{M}{p-1}-\frac{1}{p^N-1}\\
&<\frac{M}{p-1}-\frac{M}{p^{n}(p-1)}=\pi_p(M^n 0^\infty)\\
&=\pi_p(\beta_1(q)\cdots\beta_n(q)\; 0^\infty)+\pi_p(\overline{\beta_1(q)\cdots \beta_n(q)}\; 0^\infty)\\
&<\pi_p(\beta(p))+\pi_p(d)=1+\pi_p(d).\qedhere
\end{split}
\end{equation*}
\end{proof}

We are ready to prove Theorem \ref{t13}.

\begin{proof}[Proof of Theorem \ref{t13}]
In view of Lemma \ref{l31} we may assume that $q\in[q',M+1)$.

We apply the first relation of the preceding lemma with $p=q$.
Letting $n\to\infty$ and using Lemma \ref{l27} and Proposition \ref{p28} we obtain that
\begin{equation*}
\dim_H \widetilde\uu_q=\frac{h(\widetilde\uu_q')}{\log q}.
\end{equation*}
Since $\dim_H \widetilde\uu_q=\dim_H\uu_q$ and $h(\widetilde\uu_q')=h(\uu_q')$ by Lemma \ref{l25}, the equality \eqref{11} follows.
\end{proof}

\section{Proof of Theorem \ref{t14}}\label{s4}

In view of Lemma \ref{l31} it suffices to prove the theorem for $q\in[q',M+1]$.

\begin{lemma}\label{l41}
The function $D$ is left continuous in every $q\in [q',M+1]$.
\end{lemma}

\begin{proof}
Fix $q\in [q',M+1]$ and $\eps>0$ arbitrarily.
We have to show that if $p\in (1,q)$ is sufficiently close to $q$, then $\abs{D(p)-D(q)}<\eps$. The proof will be split into the following two cases.

\emph{Case I:} $q\in [q',M+1)$. Using Proposition \ref{p28} we fix a sufficiently large index $n$ such that
\begin{equation*}
h(\widetilde\vv_{q,n}')-h(\widetilde\uu_{q,n}')<\frac{\eps\log q}{2}.
\end{equation*}

Next we fix $p_n\in (1,q)$  sufficiently close to $q$, such that
\begin{equation*}
\alpha_i(p_n)=\alpha_i(q)\quad\text{for}\quad i=1,\ldots, n.
\end{equation*}

If $p\in (p_n,q)$, then using the inclusions
\begin{equation*}
\widetilde\uu_{q,n}'\subseteq  \widetilde\uu_p'\subseteq  \widetilde\uu_q'\subseteq  \widetilde\vv_{q,n}'
\end{equation*}
and applying Lemma \ref{l32} we obtain
\begin{align*}
&\frac{h(\widetilde\uu_{q,n}')}{\log p}=\dime_H\pi_p(\widetilde\uu_{q,n}')\le \dime_H\widetilde\uu_p\le \dime_H\pi_p(\widetilde\vv_{q,n}')=\frac{h(\widetilde\vv_{q,n}')}{\log p}
\intertext{and}
&\frac{h(\widetilde\uu_{q,n}')}{\log q}=\dime_H\pi_q(\widetilde\uu_{q,n}')\le \dime_H\widetilde\uu_q\le \dime_H\pi_q(\widetilde\vv_{q,n}')=\frac{h(\widetilde\vv_{q,n}')}{\log q}.
\end{align*}\
It follows that
\begin{equation*}
\begin{split}
\abs{D(p)-D(q)}&\le \frac{h(\widetilde\vv_{q,n}')}{\log p}-\frac{h(\widetilde\uu_{q,n}')}{\log q}\\
&=\frac{h(\widetilde\vv_{q,n}')-h(\widetilde\uu_{q,n}')}{\log p}+h(\widetilde\uu_{q,n}')\left(\frac{1}{\log p}-\frac{1}{\log q}\right)\\
&< \frac{\eps\log q}{2\log p}+h(\widetilde\uu_{q,n}')\left(\frac{1}{\log p}-\frac{1}{\log q}\right).
\end{split}
\end{equation*}
If $p\in (p_n,q)$ is close enough to $q$, then the right side is $<\eps$.
\medskip

\emph{Case II:} $q=M+1$. Since $D(q)=1$ and $0\le D(p)\le 1$ for all $p$, it is suffient to show that $D(p)>1-\eps$ for all $p\in (1,q)$, close enough to $q$.

Since $h(\widetilde\uu_q')=\log q=\log (M+1)>0$ by Lemma \ref{l31}, applying Proposition \ref{p28} we may fix a large integer $n$ such that
\begin{equation*}
h(\widetilde\uu_{q,n}')>\left( 1-\frac{\eps}{2}\right) \log q.
\end{equation*}

If $p\in (1,q)$ is close enough to $q$, then
\begin{equation*}
\alpha_i(p)=\alpha_i(q)\quad\text{for}\quad i=1,\ldots, n,
\end{equation*}
whence $\widetilde\uu_{q,n}'\subseteq \widetilde\uu_p'$ by \eqref{21}.
It follows that
\begin{equation*}
h(\widetilde\uu_p')>\left( 1-\frac{\eps}{2}\right) \log q.
\end{equation*}

Dividing by $\log p$ and applying Lemma \ref{l32} we infer that
\begin{equation*}
D(p)>\left( 1-\frac{\eps}{2}\right)\frac{\log q}{\log p}.
\end{equation*}
We conclude by observing that the right side is $>1-\eps$ if $p$ is close enough to $q$.
\end{proof}

We remark that for $M=1$ a simple direct proof was given for the left continuity in $q=2$ in \cite[Proposition 4.1 (i)]{DeVriesKomornik2010}.

\begin{lemma}\label{l42}
The function $D$ is right continuous in $[q',M+1)$.
\end{lemma}

\begin{proof}
Fix $q\in [q',M+1)$ and $\eps>0$ arbitrarily.
We have to show that if $p\in (q,M+1)$ is sufficiently close to $q$, then $\abs{D(p)-D(q)}<\eps$.

Using Proposition \ref{p28} we fix a sufficiently large index $n$ such that
\begin{equation*}
h(\widetilde\ww_{q,n}')-h(\widetilde\uu_{q,n}')<\frac{\eps\log q}{2}.
\end{equation*}

Next we fix $p_n\in (q,M+1)$  sufficiently close to $q$, such that
\begin{equation*}
\beta_i(p_n)=\beta_i(q)\quad\text{for}\quad i=1,\ldots, n.
\end{equation*}

If $p\in (q,p_n)$, then using the inclusions
\begin{equation*}
\widetilde\uu_{q,n}'\subseteq  \widetilde\uu_q'\subseteq  \widetilde\uu_p'\subseteq  \widetilde\ww_{q,n}'
\end{equation*}
and applying Lemma \ref{l33} we obtain that
\begin{equation*}
\dim_H\pi_p(\widetilde\uu_{q,n}')=\frac{h(\widetilde\uu_{q,n}')}{\log p}
\quad\text{and}\quad
\dim_H\pi_p(\widetilde\ww_{q,n}')=\frac{h(\widetilde\ww_{q,n}')}{\log p}.
\end{equation*}

Repeating the proof of Lemma \ref{l41} with $\widetilde\vv_{q,n}'$  changed to $\widetilde\ww_{q,n}'$, now we obtain the estimate
\begin{equation*}
\abs{D(p)-D(q)}\le \frac{h(\widetilde\ww_{q,n}') }{\log q}-\frac{h(\widetilde\uu_{q,n}')}{\log p},
\end{equation*}
and we may conclude as before.
\end{proof}

In the next result we  take any $q\in (1,\infty)$.

\begin{lemma}\label{l43}
$D$ has a bounded variation in $[q',M+1]$.
\end{lemma}

\begin{proof}
We prove that for every finite subdivision
\begin{equation*}
q_0:=q'<q_1<\cdots <q_n=M+1
\end{equation*}
the following inequality holds:
\begin{equation*}
\sum_{i=1}^n\abs{D(q_i)-D(q_{i-1})}\le \frac{2\log (M+1)}{\log q'}-1.
\end{equation*}

Writing $h(q)$ instead of $h(\uu_q')$ for brevity, we know that $h$ is non-decreasing in $[q_0,M+1]$ with $h(q_0)=0$ and $h(M+1)=\log (M+1)$.
Therefore we have the following elementary inequalities:
\begin{equation*}
D(q_i)-D(q_{i-1})
=\frac{h(q_i)}{\log q_i}-\frac{h(q_{i-1})}{\log q_{i-1}}
\le \frac{h(q_i)-h(q_{i-1})}{\log q_i}
\le \frac{h(q_i)-h(q_{i-1})}{\log q_0}
\end{equation*}
and
\begin{equation*}
D(q_i)-D(q_{i-1})
\ge \frac{h(q_{i-1})}{\log q_i}-\frac{h(q_{i-1})}{\log q_{i-1}}
\ge \frac{\log (M+1)}{\log q_i}-\frac{\log (M+1)}{\log q_{i-1}}
\end{equation*}
It follows that
\begin{equation*}
 \abs{D(q_i)-D(q_{i-1})}\le \frac{h(q_i)-h(q_{i-1})}{\log q_0}+\left( \frac{\log (M+1)}{\log q_{i-1}}-\frac{\log (M+1)}{\log q_i}\right),
\end{equation*}
and hence
\begin{align*}
\sum_{i=1}^n&\abs{D(q_i)-D(q_{i-1})}\\
&\le \frac{h(M+1)-h(q_0)}{\log q_0}+\frac{\log (M+1)}{\log q_0}-\frac{\log (M+1)}{\log (M+1)}\\
&=\frac{2\log (M+1)}{\log q_0}-1,
\end{align*}
as stated.
\end{proof}

\section{The Hausdorff dimension of $\uu$}\label{s5}

As usual, we denote by $\uu$ the set of bases $q>1$ in which $x=1$ has a unique expansion, and by $\uu'$ the set of corresponding expansions.
We recall from \cite{ErdosJooKomornik1990} and \cite{KomornikLoreti2002} that a sequence $c=(c_i)$ belongs to $\uu'$ if and only if the lexicographic inequalities
\begin{equation}\label{51}
\overline{c_1c_2\cdots}<c_{k+1}c_{k+2}\cdots<c_1c_2\cdots
\end{equation}
for all $k\ge 1$.

Fix an integer $N\ge 2$ and, inspired by the proof of \cite[Proposition 4.1 (i)]{DeVriesKomornik2010}, consider the set $\hat\uu_N'$ of sequences $c=(c_i)\in\set{0,\ldots,M}^{\infty}$ satisfying the equality 
\begin{equation*}
c_1\cdots\ c_{2N}=M^{2N-1}0,
\end{equation*}
and the  lexicographic inequalities 
\begin{equation*}
0^N<c_{kN+1}\cdots c_{kN+N}<M^N
\end{equation*}
for $k=2,3,\ldots .$ 
All these sequences satisfy \eqref{51}, so that $\hat\uu_N'\subseteq\uu'$ and $\hat\uu_N\subseteq\uu$, where we use the natural notation 
\begin{equation*}
\hat\uu_N:=\set{q\in(1,M+1]\ :\ \beta(q)\in\hat\uu_N'}.
\end{equation*}
(Here $\beta(q)$ denotes the unique and hence also greedy expansion of $x=1$ in base $q$.)

It follows from the definition of $\hat\uu_N'$ that
\begin{equation}\label{52}
\abs{B_{nN}(\hat\uu_N')}=\left((M+1)^N-2 \right) ^{n-2}\quad\text{for all}\quad  n\ge 2
\end{equation}
and 
\begin{equation}\label{53}
\abs{B_{kN+1,nN}(\hat\uu_N')}=\left((M+1)^N-2 \right) ^{n-k}\quad\text{for all}\quad n\ge k\ge 2.
\end{equation}

Consider two elements $p<q$ of $\hat\uu_N$, and let $m$ be the smallest positive integer such that $\beta_m(p)\ne\beta_m(q)$. 
Then $\beta_m(p)<\beta_m(q)$, and we deduce from the definition of $\hat\uu_N$ that
\begin{equation*}
\left( \sum_{i=1}^m\frac{\beta_i(q)}{q^i}\right) +\frac{1}{q^{m+2N}}<1<\left( \sum_{i=1}^m\frac{\beta_i(p)}{p^i}\right) +\frac{1}{p^m}\le \sum_{i=1}^m\frac{\beta_i(q)}{p^i}.
\end{equation*}
Hence 
\begin{equation*}
\frac{1}{q^{m+2N}}<\sum_{i=1}^m\beta_i(q)\left(p^{-i}-q^{-i} \right) <M\sum_{i=1}^{\infty}\left(p^{-i}-q^{-i} \right)=\frac{M(q-p)}{(p-1)(q-1)}
\end{equation*}
and therefore 
\begin{equation*}
\frac{1}{(M+1)^{m+2N}}<\frac{M(q-p)}{(q'-1)^2},
\end{equation*}
where $q'$ denotes the Komornik--Loreti constant as usual.

Setting 
\begin{equation*}
c:=\frac{(q'-1)^2}{M(M+1)^{2N}}
\end{equation*}
we conclude the following 

\begin{lemma}\label{l51}
If $p,q\in\hat\uu_N$ and $0<q-p\le c(M+1)^{-m}$ for some positive integer $m$, then $\beta_i(p)=\beta_i(q)$ for all $i=1,\ldots,m$.
\end{lemma}

Now we are ready to compute the Hausdorff dimension of $\uu$.

\begin{proof}[Proof of Theorem \ref{t16} (ii)]
Consider a finite cover $\cup I_j$ of $\hat\uu_N$ by intervals $I_j$ of length $\abs{I_j}\le c(M+1)^{-N}$.
For each positive integer $k$ we denote by $J_k$ the set of indices $j$ satisfying the inequalities 
\begin{equation*}
c(M+1)^{-(k+1)N}<\abs{I_j}\le c(M+1)^{-kN}.
\end{equation*}
We fix a large integer $n$ satisfying $c(M+1)^{-nN}<\abs{I_j}$ for all $j$; then $J_k=\varnothing$ for all $k\ge n$.

If $j\in J_k$ and $p,q\in\hat\uu_N\cap I_j$, then the first $kN$ digits of $\beta(p)$ and $\beta(q)$ coincide by the above lemma, so that at most $\abs{B_{kN+1,nN}(\hat\uu_N')}$ elements of $B_{nN}(\hat\uu_N')$ may occur for the bases $q\in\hat\uu_N\cap I_j$. 
Hence 
\begin{equation*}
\abs{B_{nN}(\hat\uu_N')}\le \sum_{k}\sum_{j\in J_k}\abs{B_{kN+1,nN}(\hat\uu_N')}.
\end{equation*}
Using \eqref{52} and \eqref{53} this is equivalent to
\begin{equation*}
\left((M+1)^N-2 \right) ^{-2}\le \sum_{k}\sum_{j\in J_k}\left((M+1)^N-2 \right) ^{-k},
\end{equation*}

Introducing the number $\sigma=\sigma(N)\in (0,1)$ by the equation 
\begin{equation}\label{54}
(M+1)^N-2=(M+1)^{\sigma N},
\end{equation}
we may rewrite the preceding inequality in the form
\begin{equation*}
(M+1)^{-2\sigma N}\le \sum_{k}\sum_{j\in J_k}(M+1)^{-\sigma Nk}.
\end{equation*}

Since 
\begin{equation*}
(M+1)^{-Nk}<c^{-1}(M+1)^N\abs{I_j}
\end{equation*}
by the definition of $J_k$, it follows that 
\begin{equation*}
(M+1)^{-2\sigma N}\le \sum_{k}\sum_{j\in J_k}c^{-\sigma}(M+1)^{\sigma N}\abs{I_j}^{\sigma}
\end{equation*}
or equivalently 
\begin{equation*}
\sum_j\abs{I_j}^{\sigma}\ge c^{\sigma}(M+1)^{-3\sigma N}.
\end{equation*}

Since the right side is positive and depends only on $N$, we conclude that $\dime_H \hat\uu_N\ge \sigma(N)$. 

It follows from the definition \eqref{54} that $\sigma(N)\to 1$ as $N\to\infty$.
Since $\hat\uu_N\subseteq\uu\subseteq\RR$ for all $N$, letting $N\to\infty$ we conclude that $\dime_H\uu=1$.
\end{proof}

\section{Proof of Theorem \ref{t15} and the Lebesgue measure of $\uu$}\label{s6}

Set $\bb':=\set{\beta(q)\ :\ q\in (1,M+1]}$ for brevity. 

Our main tool is a generalization of a reasoning in \cite{ErdosJoo1991}.
Given two positive integers $n, t$ and a word $\eta_1\cdots\eta_n\in B_n(\bb')$, the sets 
\begin{equation*}
\set{q\in [1,M+1)\ :\ \beta_i(q)=\eta_i,\quad i=1,\ldots, n}
\end{equation*}
and 
\begin{equation*}
\set{q\in [1,M+1)\ :\ \beta_i(q)=
\begin{cases}
\eta_i,\quad &i=1,\ldots, n,\\ 
0,&i=n+1,\ldots,n+t
\end{cases}
}
\end{equation*}
are two intervals $[q_1,q_2)$ and $[q_1,q_3)$ satisfying $q_3\le q_2$.

\begin{lemma}\label{l61}
The following inequality holds:
\begin{equation*}
\frac{q_3-q_1}{q_2-q_1}\ge \frac{(q_1-1)^3}{M^2q_2^{t+2}}.
\end{equation*}
\end{lemma}
We stress the fact that the right side does not depend on $n$.

\begin{proof}
It follows from the greedy algorithm that 
\begin{align}
&\sum_{i=1}^n\frac{\eta_i}{q_1^i}=1,\label{61}\\ 
&\sum_{i=1}^n\frac{\eta_i}{q_2^i}+\sum_{i=n+1}^{\infty}\frac{M}{q_2^i}\ge 1\notag
\intertext{and} 
&\left( \sum_{i=1}^n\frac{\eta_i}{q_3^i}\right) +\frac{1}{q_3^{n+t}}=1.\label{62}
\end{align}

Using the first two relations and the relation $\eta_1\ge 1$ we obtain that 
\begin{equation*}
\frac{M}{q_2^n(q_2-1)} 
\ge \sum_{i=1}^n \eta_i\left( q_1^{-i}-q_2^{-i}\right)
\ge q_1^{-1}-q_2^{-1}=\frac{q_2-q_1}{q_1q_2}.
\end{equation*}
Hence  
\begin{equation}\label{63}
(0<)q_2-q_1\le \frac{Mq_1q_2}{q_2^n(q_2-1)} .
\end{equation}

Similarly, using  \eqref{61} and \eqref{62} we obtain that 
\begin{align*}
\frac{1}{q_3^{n+t}}
&=\sum_{i=1}^n \eta_i\left( q_1^{-i}-q_3^{-i}\right)\le M\sum_{i=1}^{\infty}  \left( q_1^{-i}-q_3^{-i}\right)\\
&=M\left( \frac{q_1^{-1}}{1-q_1^{-1}}-\frac{q_3^{-1}}{1-q_3^{-1}}\right) =\frac{M(q_3-q_1)}{(q_1-1)(q_3-1)}.
\end{align*}
Hence 
\begin{equation}\label{64}
q_3-q_1\ge \frac{(q_1-1)(q_3-1)}{Mq_3^{n+t}}.
\end{equation} 

Combining \eqref{63} and \eqref{64}, and using the inequalities $q_1\le q_3\le q_2$ we conclude that 
\begin{equation*}
\frac{q_3-q_1}{q_2-q_1}\ge  \frac{(q_1-1)(q_3-1)}{Mq_3^{n+t}}\cdot\frac{q_2^n(q_2-1)}{Mq_1q_2}\ge \frac{(q_1-1)^3}{M^2q_2^{t+2}}.\qedhere
\end{equation*}
\end{proof}

In the next lemma $\lambda$ denotes the usual Lebesgue measure.

\begin{lemma}\label{l62}
The following inequality hold for all $1<p< r\le M+1$ and for all positive integers $n$ and $t$:
\begin{equation*}
\lambda\left( \set{q\in [p,r)\ :\ \beta_{n+1}(q)=\cdots=\beta_{n+t}(q)=0}\right)\ge \frac{(p-1)^3}{M^2r^{t+2}}(r-p).
\end{equation*}
\end{lemma}

Before proving the lemma we recall that the bases $q$ for which $\beta(q)$ is finite form a (countable) dense set in $[1,M+1]$. 
Indeed, if $\beta(q)$ is infinite for some $q$, then the truncated sequences $\beta_1(q)\cdots\beta_k(q)0^{\infty}$ belong to $\bb'$ for all $k=1,2,\ldots$ by an elementary reasoning given in \cite[Lemma 3.1]{KomornikLoreti2007}. 
Therefore there exist bases $q_k\in [1,M+1]$ such that 
\begin{equation*}
\beta(q_k)=\beta_1(q)\cdots\beta_k(q)0^{\infty},
\end{equation*}
and then $q_k\to q$.

\begin{proof}
We use the notations of the preceding lemma.

We may assume by density that $\beta(p)$ and $\beta(r)$ are finite. 
Choose a sufficiently large integer $n$ such that $\beta_i(p)=\beta_i(r)=0$ for all $i>n$, and consider the intervals $[q_1,q_2)$ corresponding to $n$.
Then some of these intervals form a finite partition of $[p,r)$. 
Since we have 
\begin{equation*}
\frac{q_3-q_1}{q_2-q_1}\ge \frac{(q_1-1)^3}{M^2q_2^{t+2}}\ge \frac{(p-1)^3}{M^2r^{t+2}}
\end{equation*}
for each of these intervals by the preceding lemma, the required inequality follows by summing the inequalities
\begin{equation*}
q_3-q_1\ge \frac{(p-1)^3}{M^2r^{t+2}}(q_2-q_1).\qedhere
\end{equation*}
\end{proof}

\begin{lemma}\label{l63}
Given an arbitrary real number $s>1$, there exists a sequence $(n_k)$ of natural numbers satisfying the inequalities 
\begin{equation*}
n_k>\log_s\left( n_1+\cdots+n_k\right) ,\quad k=1,2,\ldots
\end{equation*}
and the divergence relation 
\begin{equation*}
\sum_{k=1}^{\infty}s^{-n_k}=\infty.
\end{equation*}
\end{lemma}

\begin{proof}
For $s=2$ this was proved in \cite[Lemma 6]{ErdosJooKomornik1990}. 
The proof remains valid for every $s>1$.
\end{proof}

Now we are ready to prove Theorem \ref{t15}:

\begin{proof}[Proof of Theorem \ref{t15}]
By density it suffices to show for any fixed $1<p< r\le M+1$, the required property holds for almost all $q\in [p,r)$.
For convenience we normalize $\lambda$ and we use the equivalent probabilistic measure $\mu:=\frac{\lambda}{r-p}$ on $[p,r)$.
Then we may adapt the usual proof of the Borel--Cantelli lemma.

Choose a sequence $(n_k)$ satisfying the conditions of preceding lemma with $s:=r$, and set 
\begin{equation*}
C_j:=\set{q\in [p,r)\ :\ \sum_{i=n_1+\cdots+n_{j-1}+1}^{n_1+\cdots+n_j}\beta_i(q)>0},\quad j=1,2,\ldots .
\end{equation*}
It follows from Lemma \ref{l62} and \ref{l63} that 
\begin{equation*}
\mu\left( \cap_{j=k}^{\infty}C_j\right) \le \prod_{j=k}^{\infty}\left( 1-\frac{(p-1)^3}{M^2r^2}r^{-n_j}\right) =0
\end{equation*}
for every $k=1,2,\ldots .$

Therefore $C:=\cup_{k=1}^{\infty}\cap_{j=k}^{\infty}C_j$ has also zero Lebesgue measure. 
We complete the proof by observing that if $q\in [p,r)\setminus C$, then $\beta(q)$ has the required property for infinitely many $m=n_1+\cdots+n_k$.
\end{proof}

Finally we compute the Lebesgue measure of $\uu$:

\begin{proof}[Proof of Theorem \ref{t16} (i)]
Since $\uuu\setminus\uu$ is countable, it suffices to prove that $\uu$ is a null set.
Furthermore, it suffices to prove that $\uu\cap[p,M+1)$ is a null set for each $p\in (1,M+1)$ such that $\beta(p)$ is finite.

It follows from the lexicographical characterization \eqref{51} of $\uu$ that $\uu\cap[p,M+1)\subseteq C$, where $C$ is the null set in the proof of the above lemma, corresponding to the choice $[p,r)=[p,M+1)$. 
Hence $\uu\cap[p,M+1)$ is a null set indeed.
\end{proof}

\section{Proof of Theorem \ref{t17}}\label{s7}

In view of Theorems \ref{t11} and \ref{t14} it suffices to prove that $D'<0$ almost everywhere in $(q',\infty)$.
This was implicitly proved in \cite[Theorems 2.5 and 2.6]{KongLi2014}.
Here we give an alternative proof.

Since $\uuu$ is a null set by Theorem \ref{t15} (i), it suffices to prove that $D'<0$  everywhere in each connected component $I=(q_0,q_0^*)$ of $(q',\infty)\setminus\uuu$. 
Fixing $p\in (q_0,q_0^*)$ arbitrarily, we deduce from Theorem \ref{t13} and Lemma \ref{l211} that 
\begin{equation*}
D(q)=\frac{h(\uu_p')}{\log q}
\end{equation*}
for all $q\in I$, and therefore 
\begin{equation*}
D'(q)=-\frac{h(\uu_p')}{q(\log q)^2}
\end{equation*}
for all $q\in I$. 
Since $p>q'$ and therefore $h(\uu_p')>0$ by Theorem \ref{t11}, we have $D'(q)<0$ for all $q\in I$ indeed. 

\begin{remark}
Since $q'$ and $M+1$ are the smallest and largest elements of $\uuu$, the first and last connected components of $(1,\infty)\setminus\uuu$ are $(1,q')$ and $(M+1,\infty)$.

We recall from \cite{DeVriesKomornik2009} that the left and right edpoints of the remaining connected components $I=(q_0,q_0^*)$ run over $\uuu\setminus\uu$ and some proper subset $\uu^*$ of $\uu$, respectively.

It follows from some theorems of Parry \cite{Parry1960} and Solomyak \cite{Solomyak1994} that each element of $\uuu\setminus\uu$ is an algebraic integer, all of whose conjugates are smaller than the Golden Ratio in modulus.

On the other hand, it was proved in \cite{KongLi2014} that the points $q_0^*$, called \emph{de Vries--Komornik numbers}, are transcendental.
The smallest one is the \emph{Komornik--Loreti constant} $q'$.
Their expansions are closely related to the classical Thue--Morse sequence.
\end{remark}

\section*{acknowledgments}
The second author is supported by the National
Natural Science Foundation of China no  11401516 and JiangSu Province Natural
Science Foundation for the Youth no BK20130433. The third author is supported by the
National Natural Science Foundation of China no 11271137.

\end{document}